\documentclass [12pt]{amsart}
\usepackage{amsmath}
\usepackage[foot]{amsaddr}
\usepackage{pdfsync}
\usepackage{amsthm}
\usepackage{graphicx}
\usepackage{epsfig}
\usepackage{caption}
\usepackage{tikz}
\usepackage{amsfonts, amssymb}
\usepackage[utf8]{inputenc} 
\usepackage[english]{babel}
\usepackage[numbers]{natbib}

\newtheorem{theorem}{Theorem}
\newtheorem{lemma}[theorem]{Lemma}
\newtheorem{proposition}[theorem]{Proposition}
\newtheorem{corollary}[theorem]{Corollary}
\theoremstyle{definition}
\newtheorem{remark}[theorem]{Remark}
\theoremstyle{definition}
\newtheorem{example}[theorem]{Example}
\theoremstyle{definition}
\newtheorem{definition}[theorem]{Definition}

\setcitestyle{aysep={ }}

\def\R{{\mathbb R}}

\def\SS{{\mathbb S}}

\def\FF{{\mathcal Q}}
\def\GG{{\mathcal G}}

\begin{document}
\title{ Simple game induced manifolds}
\author{Pavel Galashin$^1$}
\email{pgalashin@gmail.com}
\address{$^1$Department of Math. and Mech., St. Petersburg State University}
\author{Gaiane Panina$^2$}
\email{gaiane-panina@rambler.ru}
\address{$^2$SPIIRAS, St. Petersburg State University}

\maketitle

\begin{abstract}
 Starting by a simple game $\FF $  as a combinatorial data,
we build up a cell complex  $M(\FF)$, whose construction resembles
combinatorics of the permutohedron. The cell complex proves to be a
combinatorial manifold; we call it the \textit{ simple game induced
manifold.} By some motivations coming from polygonal linkages,  we
think of $\FF$ and of  $M(\FF)$ as of\textit{ a  quasilinkage} and
the \textit{moduli space of the quasilinkage}  respectively.
  We present some examples of quasilinkages and show that the moduli space
  retains many  properties of moduli space of polygonal linkages.
  In particular, we show that the moduli space $M(\FF)$ is homeomorphic to
   the space of stable point configurations on $S^1$, for an associated with a
  quasilinkage notion of stability.

\end{abstract}


%

\section{Introduction}\label{SectionIntroduction}

It is a usual praxis that some combinatorial data produce a
geometric object. Classical examples are \textit{permutohedron},
\textit{associahedron} (see \citep{Z}), other ``famous'' polytopes,
and their  generalizations \textit{graph-associahedra}  and
\textit{nestohedra} (see \citep{Post}). In the paper, we act in a
somewhat similar way starting by a \textit{simple game} $M(\FF)$ (in
the  usual sense of game theory) as a combinatorial data.
 We build up a cell complex  $M(\FF)$, whose construction although resembles very much the combinatorics of the permutohedron, yet depends on the simple game.  The cell complex proves to be a combinatorial manifold, which we call the
 \textit{simple game induced manifold}.

The idea is borrowed from the cell decomposition of the moduli space
of polygonal linkages (see \citep{P}). This motivates us to treat a
simple game $\FF$ as \textit{a quasilinkage} since it provides a
natural generalization of polygonal linkages.

By the same reason, we call the  cell complex $M(\FF)$ the
\textit{moduli space of the quasilinkage.} The paper presents the
basic study of the  simple game generated manifolds.

\subsection*{Polygonal linkages:  definitions and overview of the results }
\label{subsect:prelim} Given a vector $L=(l_1,...,l_n)\in \R_+^n$ of
$n$ positive real numbers, consider $n$ rigid bars of lengths
$l_1,...,l_n$ joined in a {\it closed chain}. Such a construction is
called a {\it polygonal linkage}. By $M(L)$ we denote its\textit{
moduli space}, or \textit{the space of planar configurations}:

$$
M(L):=\{z_1,...,z_n \in \mathbb{R}^2: |z_i|=1,\ \sum l_iz_i=0\} /
SO(2)$$
$$=\{z_1,...,z_n \in \mathbb{R}^2: |z_i|=1,\ \sum l_iz_i=0,\ z_1=1\} \ .
$$

%

Denote by $[n]$ the set $\{1,...,n\}$.

\begin{definition}The length vector $L$ is called {\it generic}, if there is no
subset $J\subset [n]$ such that
$$\sum_{i \in J} l_i = \sum_{i\notin J} l_i.$$
Throughout the paper, we consider only generic length vectors $L$.
\end{definition}
 The hyperplanes
$$\sum_{i \in J} l_i = \sum_{i\notin J} l_i$$
called {\it walls} subdivide $\R_+^n$ into a collection of {\it
chambers}.

\bigskip
 Here is a (far from complete) summary of facts
about $M(L)$:

 \begin{itemize}
\item For a generic  length  vector, M(L) is a smooth manifold \cite{KM5}.
 \item The topological type of $M(L)$ depends only on the chamber of $L$  \cite{KM5}.
 \item  As it was shown in \citep{P}, $M(L)$ admits a structure of a regular cell complex.
  The combinatorics is very much related (but not equal) to the combinatorics of the permutahedron. The construction will be explained in details in Section \ref{sect:manifold}.
 \end{itemize}

\begin{definition}
\label{dfn:short} For a generic length vector $L$, a subset
$J\subset [n]$ is called {\it long}, if $$\sum_{i \in J} l_i >
\sum_{i\notin J} l_i.$$ Otherwise, $J$ is called \textit{short}. The
set of all short sets we denote by $\mathcal{S}(L)$.
\end{definition}
\begin{itemize}
 \item  Homology groups of $M(L)$ are free
 abelian groups. For a generic length vector $L$,
 the rank of the homology group $H_k(M(L))$  equals  $a_k+a_{n-3-k}$, where $a_i$ is the number of short
subsets of size $i+1$ containing the longest edge (see \citep{FS}).
\end{itemize}

We stress that the manifold  $M(L)$  (considered either as a
topological manifold, or as a cell complex) is uniquely defined by
the collection of short subsets of $[n]$.

\bigskip

\subsection*{Quasilinkages}
The following definition generalizes Definition \ref{dfn:short}.
\begin{definition}
A family $\FF$ of subsets of $[n]$ is called {\it a quasilinkage}, if it satisfies the following properties:
\begin{enumerate}
\item $\FF$\textbf{ contains all singletons}:  for any $i \in [n]$, $\{i\} \in \FF$.
\item \textbf{Monotonicity}: if $S\in \FF$, and $T\subset S$ then $T\in \FF$.
\item \textbf{Strong complementarity}: if $S\in \FF$ then $([n]\setminus S)\notin \FF$ , and, conversely, if $S\notin \FF$, then $([n]\setminus S)\in \FF$.
\end{enumerate}
\end{definition}

The proposed notion exists in the literature; yet in  completely
different frameworks. It appeared  as ``simple game with constant
sum'' in game theory, see \citep{NM,E} and also as ``strongly
complementary simplicial complex'', see \citep{B,B2}.

Following the aforementioned motivation by polygonal linkages, we
call any $S\in \FF$ {\it a $\FF$-short set}, or simply \textit{ a
short set}, and any $S\notin \FF$ {\it a long set}.

\begin{remark}
 Each polygonal linkage $L$ yields a quasilinkage by the above defined short sets family $\mathcal{S}(L)$ (see Definition \ref{subsect:prelim}).
\end{remark}

\begin{definition} A quasilinkage $\FF$ is called {\it real}, if there exists a
length vector $L$ such that $\mathcal{S}(L)=\FF$. Otherwise,  $\FF$
is called {\it imaginary}.

\end{definition}

Here we list some additional properties that are true for real
quasilinkages, but in general may not hold for imaginary ones:
\begin{enumerate}
\item \textbf{Comparability}: For any  $A,B \in2^{[n]}$, and any $i,j\notin A\cup B$, if $A\cup i$ is long, $A\cup j$ is short and $B\cup i$ is short then $B\cup j$ is also short.
The property means that the edge $i$ is in a sense "longer" than
 $j$.
\item \textbf{Trade robustness}: Given $k$ long subsets, there is no interchanging of the elements of these sets, which makes all of them short.
\end{enumerate}
There arises a natural  question: given a family of subsets $\FF$,
under what conditions there exists a length vector $L$ such that
$\FF=\mathcal{S}(L)$? This question has been studied a lot in game
theory. The family of subsets with the monotonicity property is
called {\it a simple game}, and if there exists a corresponding
length vector, then this family is called {\it a weighted majority
game}. In \citep{TZ} it was shown, that a simple game is a weighted
majority game if and only if it satisfies the trade robustness
condition. Other characterizations of weighted majority games are
given in, for example, \citep{NM,E}.

In our terminology, the trade robustness condition guarantees that a
given quasilinkage is real.

\subsection*{Main results}

  We start with small examples of
imaginary quasilinkages. Next, we give two ways of cooking  up
quasilinkages: the \textit{flip technique} and the
\textit{conflict-free family extensions} (Section
\ref{sect:imaginary}). This implies that the class of all
quasilinkages is much wider than the class of all linkages. Yet more
examples arise via oriented matroid approach in Section
\ref{sect:matroids}.

In Section \ref{sect:manifold} we  associate with a quasilinkage
$\FF$  a cell complex  $CWM(\FF)$ by applying the rules from
\citep{P}. We prove that $CWM(\FF)$ is locally isomorphic to
$CWM(L)$ for some real linkage $L$ (however, $L$ depends on the
location, and  there may be no real linkage associated to the entire
complex).

As a corollary, we immediately see that $CWM(\FF)$ is  a manifold of
dimension $n-3$.

In Section \ref{sect:stable} we show  that the manifold $CWM(\FF)$
is homeomorphic to the moduli space of stable point configurations
on $\SS^1$ for an appropriate definition of stability.
%

\section{Imaginary quasilinkages}
\label{sect:imaginary}
\subsection{Small symmetric examples and non-examples}
 Elementary case analysis shows that for $n \leq 5$ there are no imaginary quasilinkages.
However,  for $n\ge 6$ there are  many. We start with some symmetric
examples of imaginary quasilinkages in low dimensions.

\begin{definition}
We say that a quasilinkage $\FF$ is {\it symmetric} if for any $i,j
\in [n]$ there exists an element $\sigma$ of the symmetric group
$S_n$ such that:
\begin{enumerate}
\item $\sigma$ takes $i$ to $j$, and
\item $\sigma$ takes short sets to short sets. (Equivalently, if $\sigma$ takes long sets to long
sets.)

\end{enumerate}
\end{definition}

\begin{example}\citep{NM}
\label{ex6} Let $n=6$. A symmetric quasilinkage is defined  by the
following rules:
\begin{enumerate}
\item All $2$-element sets are short. (Equivalently, all $4$-element sets are
long.)
\item The only ten short $3$-element subsets are:
$$123,124,135,146,156,236,245,256,345,346.$$
\end{enumerate}
\end{example}


We give another example for $n=7$, which is also symmetric:
\begin{example}\citep{NM}
\label{ex7} A symmetric quasilinkage for $n=7$ is defined as
follows:
\begin{enumerate}
    \item All $2$-element subsets are short.
    \item The only seven $3$-element long subsets are:
$$123,145,167,257,246,347,356.$$
\end{enumerate}
\end{example}
Example \ref{ex7} actually corresponds to Fano plane, and its
automorphism group is known to be transitive, so this example is
again symmetric.

Example \ref{ex6} corresponds to the $6$-vertex triangulation of
projective $2$-plane, and can be generalized as vertex-minimal
triangulation of projective space only in dimensions $4,8,16$, see
\citep{B,B2}.

\begin{lemma}
\label{lemma:symmetric}
  \begin{enumerate}
   \item If $n$ is odd, there exists
exactly one symmetric real linkage. It assigns equal lengths to all
the edges. Equivalently, a set is short whenever its size is smaller
than $n/2$.
\item If $n$ is even, there exists no symmetric real  linkage.
  \end{enumerate}
\end{lemma}
\begin{proof}
  Fix any symmetric real quasilinkage $\FF$ with length vector $L$. For  $j\in [n]$ and    $k\in\mathbb{ N}$, denote by $a_k(j)$ the number of short subsets of size $k+1$ containing $j$.
  By symmetry assumption,  $a_k(j)$ does not depend on $j$. Now assume that $l_i<l_j$ for some $i,j\in [n]$. Take
  a set $A\subset[n]$ such that $i,j\notin A$. If $A\cup j$ is short, then $A\cup i$ is also short. If $A\cup i$ is short and $A\cup j$ is long, then $a_{|A|}(i)>a_{|A|}(j)$,
   which contradicts the symmetry assumption. Therefore $A\cup j$ is short if and only if $A\cup i$ is short for any $i,j\in [n]$. This means that for any $k$,
   all the $k$-element subsets of $[n]$ are either simultaneously short or simultaneously long. This immediately implies the result of the lemma.
\end{proof}

\begin{lemma}
\label{lemma:symmetric}
  \begin{enumerate}
   \item If $n$ is odd, there exists
exactly one symmetric real linkage. It assigns equal lengths to all
the edges. Equivalently, a set is short whenever its size is smaller
than $n/2$.
\item If $n$ is even, there exists no symmetric real  linkage.
  \end{enumerate}
\end{lemma}
\begin{proof}
  Fix any symmetric real quasilinkage $\FF$ with length vector $L$. For  $j\in [n]$ and    $k\in\mathbb{ N}$, denote by $a_k(j)$ the number of short subsets of size $k+1$ containing $j$.
  By symmetry assumption,  $a_k(j)$ does not depend on $j$. Now assume that $l_i<l_j$ for some $i,j\in [n]$. Take
  a set $A\subset[n]$ such that $i,j\notin A$. If $A\cup j$ is short, then $A\cup i$ is also short. If $A\cup i$ is short and $A\cup j$ is long, then $a_{|A|}(i)>a_{|A|}(j)$,
   which contradicts the symmetry assumption. Therefore $A\cup j$ is short if and only if $A\cup i$ is short for any $i,j\in [n]$. This means that for any $k$,
   all the $k$-element subsets of $[n]$ are either simultaneously short or simultaneously long. This immediately implies the result of the lemma.
\end{proof}

\begin{corollary} Examples \ref{ex6} and \ref{ex7} present imaginary
quasilinkages.
\end{corollary}

\begin{proposition}
  For $n=8$, there is no symmetric quasilinkage (neither real, no imaginary).
\end{proposition}
\begin{proof}
  There are ${8\choose 4}=70$ four-element subsets of $[n]$. For any quasilinkage, exactly $35$ of them are long, and $35$ of them are short.
  By symmetry, any of the $8$ elements of $[n]$ should be contained in the same number of short $4$-element subsets, therefore $35\cdot4$ should be divisible by $8$, but it is not. 
\end{proof}

\bigskip

\subsection*{Flips}
\label{subsect:flips}
\begin{definition}
  Let $\FF$ be a quasilinkage, and let $T$ be a maximal (by inclusion) subset of $[n]$ such that $T\in\FF$. Define the {\it flip} $F_T(\FF)$ as follows:
  $$F_T(\FF):=(\FF\setminus \{T\}) \cup \{([n]\setminus T)\}$$
\end{definition}
In other words, a flip is an operation that makes the $\FF$-short set $T$ long, and its complement short, leaving all the other sets unchanged.

\begin{proposition}
$F_T(\FF)$ is again a quasilinkage.
\end{proposition}
\begin{proof}
   The strong complementarity property  obviously holds for $F_T(\FF)$, so it remains to check monotonicity  for $F_T(\FF)$.
    Assume that $S\subset S'\subset [n]$, and $S'\in F_T(\FF)$. We need to prove that $S\in F_T(\FF)$. If $S'\neq \overline{T}:=([n]\setminus T)$ then every  proper subset  of $S'$  is
     $\FF$-short and is not equal to $T$ by maximality, so the only remaining case is $S'=\overline{T}$.
     But every proper subset of $\overline{T}$ is $\FF$-short, again, by maximality of $T$, so the proposition is proven.
\end{proof}

 \begin{example}
 Take the length vector $L=(l_1,...,l_6)$ with
 $$\l_1=l_2=l_3=1+\varepsilon,\ l_4=l_5=l_6=1.$$
 It corresponds to a real quasilinkage $\mathcal{S}(L)$. Now take the (maximal short) set $T=\{4,5,6\}$ and make a flip $\FF:=F_T(\mathcal{S}(L))$. This quasilinkage is imaginary, because it violates the comparability condition: $\{4,5,6\}$ is $\FF$-long, while $\{1,5,6\}$ is $\FF$-short, so $4$ must be longer than $1$, but, from the other hand, $\{1,3,5\}$ is $\FF$-long, while $\{4,3,5\}$ is $\FF$-short.
  \end{example}

  This example differs from Example \ref{ex6}.
  One more example of an imaginary quasilinkage arises from the below
  proposition.

  \begin{proposition}
  Any flip of an imaginary quasilinkage $\FF$ from Example \ref{ex6} is again imaginary.
  \end{proposition}
  \begin{proof}
    Because of the total symmetry of $\FF$, it does not matter what set we will choose to be flipped, so we can choose $T:=\{1,2,3\}$. But the quasilinkage $\GG:=F_T(\FF)$ still violates the comparability condition: the sets $\{1,2,4\}$ and $\{3,4,5\}$ are $\GG$-short while the sets $\{3,2,4\}$ and $\{1,4,5\}$ are $\GG$-long, so $1$ and $3$ are not comparable.
  \end{proof}
%

  \begin{proposition}
  \label{prop:flipconnected}
   For a fixed $n$, any two $n$-quasilinkages are connected by a sequence of flips.
  \end{proposition}
  \begin{proof}
    Take an arbitrary quasilinkage $\FF$, and take any maximal short set $T\subset[n]$ such that $1\in T$. Apply the flip $F_T(\FF)$, take any other maximal short set
     containing $1$, and make it long by another flip, and so on. After a finite number of steps we get a quasilinkage $\FF'$ such that the set $S$ is $\FF'$-long if
     and only if it contains $1$. This quasilinkage corresponds to the real quasilinkage $\mathcal{S}(L)$ for the length vector $L=(1,\varepsilon,\varepsilon,...,\varepsilon)$.
  \end{proof}

\bigskip

\subsection*{Conflict-free family extensions}
\begin{definition}
A family $\GG$ of subsets of $[n]$ is called {\it conflict-free}, if for any $T,S\in\GG$ it is true that $([n]\setminus T)\not\subset S$
\end{definition}
A conflict-free family represents our partial knowledge about which
sets are short and which sets are long, and every short set doesn't
contain any long subsets.

A subset $S\subset [n]$ is called {\it $\GG$-unknown} if  neither
$S$, nor its complement $\overline{S}$ is contained in an element of
$\GG$.

\begin{lemma}
Any conflict-free family of subsets $\GG$ extends to a quasilinkage,
i.e., there exists a quasilinkage $\FF$ such that  $\GG \subset
\FF$.
\end{lemma}
\begin{proof}
Let $S$ be some $\GG$-unknown subset. Then  $\GG':=\GG \cup \{S\}$
is again conflict-free. This means that we can add $\GG'$-unknown
subsets one by one until  unknown subsets exist. Finally, we arrive
at a conflict-free family of subsets $\GG''$,with no $\GG''$-unknown
subsets. It is a desired quasilinkage.
\end{proof}

So, now we have a way of constructing  imaginary quasilinkages:
 \begin{enumerate}
   \item Start with some small conflict-free family $\GG$, which cannot be  incorporated into any real linkage. For instance, one can take a set violating the  comparability property.
   \item Add one by one all $\GG$-unknowns.
   \item The result will be automatically an imaginary  quasilinkage.
 \end{enumerate}

For example, let $n=6$, and let $\GG=\{123,356,245,146\}$. If these
subsets are short, then the subsets $124$ and $235$ are long,
whereas $123,245$ are short. It is a conflict-free family which
doesn't satisfy the comparability property (therefore, imaginary).

\begin{definition}(Freezing for quasilinkages)
Assume that $S_1,...,S_k$ is a (non-ordered) partition of $[n]$ into
$k$ non-empty short sets. We build a new quasilinkage $FREEZE(\FF)$
on the set $[k]$ by the rule:

$$J\subset [k] \hbox{ is short iff }\bigcup _{i\in J}S_i \hbox{ is
short.}$$
\end{definition}

\section{Moduli space of a quasilinkage}
\label{sect:manifold}
\subsection*{Cell structure on the moduli space of a real linkage: a reminder}
Fix   a generic length vector $L$. We remind that to describe a
regular cell complex, it suffices to list all the (closed) cells
ranged by dimension, and to describe incidence relations for closed
cells.

\begin{definition}
A cyclically ordered partition $S_1,...,S_k$ of $[n]$ into $k$ non-empty subsets is called {\it admissible}, if every $S_i,1\le i\le k$, is short.
\end{definition}

\begin{theorem}\label{TheoremPaninaCellComplex}\citep{P}
The below described  cell complex $CWM^*(L)$ is a combinatorial
manifold homeomorphic to the  moduli space $M(L)$.
 \begin{enumerate}
   \item  The $k$-cell of the complex are labeled by (all possible) admissible cyclically ordered partition of $[n]$ into $(n-k)$ non-empty subsets. Given a cell $C$, its label is denoted by $\lambda (C)$.
   \item A closed cell $C$  belongs to the boundary of another closed cell $C'$ whenever the label $\lambda (C')$ is finer than the label $\lambda
   (C)$.\qed
 \end{enumerate}

\end{theorem}

We stress  that the complex $CWM^*(L)$ depends  only on the family
of short subsets $\mathcal{S}(L)$. This hints that this construction
can be extended  to  quasilinkages.

\subsection*{Cell complex associated to a quasilinkage}
Assume that a quasilinkage $\FF$  is fixed. Although the notion of
(planar) configurations  has no sense, we can literally repeat the
construction of the above cell complex.

\begin{definition}
A cyclically ordered partition $S_1,...,S_k$ of $[n]$ into $k$
non-empty subsets is called {\it $\FF$-admissible}, if every
$S_i,1\le i\le k$, is  $\FF$-short.
\end{definition}

\begin{definition}
For a quasilinkage $\FF$ it's {\it moduli space} $M(\FF)$ is the cell complex
defined as follows:
 \begin{enumerate}
   \item  The $k$-cell of the complex are labeled by (all possible) admissible cyclically ordered partition of $[n]$ into $(n-k)$ non-empty subsets. Given a cell $C$, its label is denoted by $\lambda (C)$.
   \item A closed cell $C$  belongs to the boundary of another closed cell $C'$ whenever the label $\lambda (C')$ is finer than $\lambda (C)$.
 \end{enumerate}
\end{definition}

The complex is a combinatorial manifold, which is locally isomorphic
to the complex $CWM^*(L)$ of some real linkage:

\begin{theorem}
\label{thm:manifold}\begin{enumerate}
                      \item For every vertex $v$ of cell complex $M(\FF)$,
                       there exists a length vector $L_v$ such that the star of the vertex $v$ is combinatorially isomorphic to the star of some vertex of $CWM^*(L)$.
                       \item For every cell $\sigma$ of cell complex $M(\FF)$,
                       there exists a length vector $L_\sigma$ such that the star of the cell $\sigma$ is combinatorially isomorphic to the star of some vertex of $CWM^*(L_\sigma)$.
                      \item For every quasilinkage $\FF$, the complex $M(\FF)$ is a combinatorial manifold.

                    \end{enumerate}

\end{theorem}
\begin{proof}
(1) Fix a vertex $v$ of $M(\FF)$. By construction, it is labeled by
some $\FF$-admissible cyclically ordered partition of $[n]$ into $n$
short non-empty subsets, that is, by a cyclic  ordering on $[n]$.
Without loss of generality we may assume that $v$ is labeled by the
partition
$$\lambda(v)=\{1\},\{2\},...,\{n\}.$$

The partition $p$ should be  viewed as numbers $1,...,n$ placed on the circle counterclockwise.

We need the following observation: let $\sigma$ be a $k$-cell of
$M(\FF)$ labeled by a partition $\lambda=S_1,...,S_{n-k}$. Then
$\sigma$ is incident to $v$ if and only if each of the sets  $S_i$
is of the form $\{a,a+1,...,a+b\}$ for some natural numbers $a$ and
$b$ (the sums are taken modulo $n$). It is true because otherwise
the partition $\lambda(v)$ would not be a refinement of $S$. Let us
call the sets of the form $\{a,a+1,...,a+b\}$  {\it the segments of
the partition $\lambda(v)$}.

Now (1) follows from the lemma:
\begin{lemma}
In the above notation, there exists a length vector $L_v$ (depending
on the vertex $v$) such that for any segment $T$ of the partition
$\lambda(v)$, the set $T$ is $\FF$-short if and only if $T$ is
$L_v$-short.
\end{lemma}
{\it Proof of the lemma.}

To construct such a length vector, we will need some additional
observations. Recall that $\lambda(v)$ is viewed as numbers
$1,...,n$ placed on the circle. There are $n$ ways to break the
circle into a line: $1,2,...,n$, $2,3,...,n,1$, ...,
$n,1,2,...,n-1.$
 We will call such way  {\it a separator position}. Take a separator position $s$, for example, $2,...,n,1$.
 There exists a unique number $q=q(s)\in [n]$, such that the set $\{2,3,...,q-1\}$ is short, and the set $\{2,3,...,q\}$ is long.
 We analogously define $q(s)$ for all separators $s'$.

We are now ready to define the length vector. For any $j\in [n]$ put $l_j:=1+|q^{-1}(j)|$. Equivalently speaking,
$$l_j:=1+\frac12|\{S\subset [n]: S \textrm{ is a short segment of }p; S\cup\{j\} \textrm{ is a long segment of } p\}|.$$

It is clear that the total length of all edges is always equal to $2n$. We need to prove that the segment $S$ of $p$ is short iff $\sum_{j\in S}l_j<n$. Note that $\sum_{j\in S}l_j=|S|+|q^{-1}(S)|$.

Take arbitrary short segment $S$ of $p$. If $s$ is a separator position adjacent to some element of $S$ (there are $|S|+1$ such separator positions), then it is obvious that $q(s)\notin S$. Therefore $|q^{-1}(S)|\leq n-|S|-1$, because the total number of separator positions equals to $n$. So for short segment $S$ of $p$ we conclude that $\sum_{j\in S}l_j=|S|+|q^{-1}(S)|\leq n-1$. Lemma is proven. \qed

(2) The star of a cell  can be reduced to the case (1) by freezing
technique. Indeed, for a cell $\sigma$ labeled by
$\lambda(\sigma)=S_1,S_2,...,S_k,$ we freeze all the entries in each
of the sets $S_i$, and arrive at a quasilinkage on the set $[k]$.

(3)  follows directly from (1), (2), and  Theorem
\ref{TheoremPaninaCellComplex}. 
\end{proof}

The below construction gives an analysis of the vertex stars of the
complex  $M(\FF)$.

Assume  that a quasilinkage $\FF$ and a vertex $v$ of $M(\FF)$ are
fixed. Theorem \ref{thm:manifold} assigns to $v$ a length vector
$L_v=(l_1,...,l_n)$. Without loss of generality we may assume that
$l_1+...+l_n=2\pi$ and that $v$ is labeled by the cyclical ordering
$\lambda(v)=(1,2,...,n)$.

Decompose the (metric) circle $S^1$ centered at the origin $0$ into
a union of arches of lengths $l_1,..,l_n$. The endpoints of the
arches give the \textit{Gale diagram} (see \citep{Z}) of   some
convex polytope $K=K(F,v)\subset\R^{n-3}$.

\begin{proposition}
The star of the vertex $v$ is combinatorially dual to the above defined convex polytope $K$.
\end{proposition}
\begin{proof}
The vertices of  $K$ correspond to partitions of $[n]$ into $n-1$
short subsets, and, equivalently, to the short pairs of the form
$(i,i+1)$ (this pair is represented by the vector $u_i$. By a
property of Gale diagrams, the vertices of the set $I\subset [n]$
form a facet if and only if the convex hull $conv(\{u_i|i\in
([n]\setminus I)\}$ contains the origin $0$ the in its relative
interior. This means that the angle between every two succeeding
vectors of the set $([n]\setminus I)\}$ is smaller than $\pi$. Let
the indices $i_1,i_2\notin I$ be such that for any $i_1<i<i_2$, we
have  $i\in I$. Then the angle between $u_{i_1}$ and $u_{i_2}$ is
equal to the sum $\sum_{i_1<i\le i_2} l_i$. So the vertices of the
set $I$ form a facet if and only if $I$ gives a refinement of
partition $\lambda(v)$ into short subsets. This corresponds to the
cell incident to $v$, which completes the proof of the proposition.
\end{proof}

\begin{theorem}  For any quasilinkage $\FF$,
 the complex $M(\FF)$  admits a PL structure.
\end{theorem}
Proof. We refer the reader to literally the  same proof  of the
analogous theorem for real linkages from  \citep{P}. In short, each
cell is combinatorially equivalent to a Cartesian product of
permutohedra. We metrically realize each of the cells by the
Cartesian product of standard permutohedra. Then the gluing map is
an isometry.\qed

\bigskip

The next proposition gives us  information about what happens to the
moduli space of after a flip (see subsection \ref{subsect:flips}).

\begin{proposition}
\label{prop:flipsurgery}
  Let $\FF$ be a quasilinkage, and let $T$ be any maximal $\FF$-short subset of $[n]$. Then the moduli space
  of the flipped quasilinkage
   $M(F_T(\FF))$ differs from $M(\FF)$ by a Morse surgery of index $(n-|T|-1)$.
\end{proposition}
\begin{proof}

  Consider the cell complex $M(\FF)$. The flip deletes from the complex some of the cells and adds some new cells. Assume that a cell
  labeled by some partition $S=(S_1,...,S_k)$ gets deleted. This means that $T\subseteq S_i$ for some $i$.
    Since $T$ is a maximal $\FF$-short set, we have $T=S_i$.
    Therefore, all the $(n-k)$-cells which are deleted during the flip are labeled by all possible partitions of type $(T,S_1,S_2,...,S_{k-1})$.
    Thus we arrive at the cell structure of the boundary of the permutohedron (see \citep{Z}) $\Pi_{n-|T|}\subset \R^{n-|T|-1}$
    multiplied by a disk.
     The cell structure of $M(\FF)$ converts this disk to the permutohedron $\Pi_{|T|}$.
    So, we cut out a cell subcomplex $(\partial \Pi_{n-|T|})\times \Pi_{|T|}$ and then we patch instead the cell complex $\Pi_{n-|T|}\times \partial \Pi_{|T|}$ along the identity mapping on
     their common boundary $\partial \Pi_{n-|T|}\times \partial \Pi_{|T|}$. This operation is the Morse surgery of index $(n-|T|-1)$.
\end{proof}

\begin{remark}
  Propositions \ref{prop:flipsurgery} and \ref{prop:flipconnected} give an alternative proof
  of Theorem \ref{thm:manifold}.
\end{remark}

%
%
%

%

\section{Stable point configurations}
\label{sect:stable}
There is an important relationship between moduli space of a polygonal linkage and moduli space of stable point configurations on $S^1$. The relationship almost automatically extends to quasilinkages.
 We stress that the below is a combination of the classical construction borrowed from \citep{KM5}
 with the cell decomposition approach from \citep{P}.

 \bigskip

Assume that a quasilinkage $\FF$ is fixed.
 \begin{definition}A  configuration of $n$  (not necessarily distinct) marked  points
 $p_1,...,p_n $ on the unit circle $ \SS^1$ is called   $\FF${\it-stable}
 if the following holds:

  If  the points $\{p_i\}_{i \in I}$ coincide, then the set $I \subset [n]$  is $\FF$-short.
 \end{definition}

 We identify $\SS^1$ with the real projective line  $\mathbb{R}P^1$, which enables us to speak of diagonal action of the
 group $PSL(2,\mathbb{R})$  on the space of all stable configurations.
 We introduce  the quotient space

 $$M_{st}(\FF)= \{\hbox{space of }  \FF\hbox{-stable configurations}\}/PSL(2,\R).$$

 \begin{theorem}
 Given a quasilinkage $\FF$,
 \begin{enumerate}
              \item $M_{st}(\FF)$ is a $(n-3)$-dimensional manifold.
              \item $M_{st}(\FF)$ is homeomorphic to  $M(\FF)$.
              \item The stratification of the space $M_{st}(\FF)$  by combinatorial types is a regular cell complex dual to  the cell complex $M(\FF)$.
            \end{enumerate}

 \end{theorem}

\begin{proof}
We label each point configuration  by its \textit{combinatorial type} -- the cyclically ordered partition of the set $[n]$. The labels do not change under the action of the group $PSL(2,\R)$.
 Equivalence classes are open balls of different dimensions, and  can be considered as open cells of some cell decomposition.

We arrive at the cell complex  on $M_{st}(\FF)$
defined as follows:
 \begin{enumerate}
   \item  The $k$-cell of the complex are labeled by (all possible) admissible cyclically ordered partition of $[n]$ into $k+3$ non-empty subsets. Given a cell $C$, its label is denoted by $\lambda (C)$.
   \item A closed cell $C$  belongs to the boundary of another closed cell $C'$ whenever the label $\lambda (C')$ is finer than $\lambda (C')$.
 \end{enumerate}
 This  cell decomposition
 is obviously combinatorially dual to the cell complex $M(\FF)$.
\end{proof}

\section{A family of quasilinkages generated by an oriented matroid}
\label{sect:matroids}
\subsection*{Oriented matroids: a short reminder}

Let us start with some definitions.
\begin{definition}
 A  $(n-1)$-pseudosphere is a tame embedding of the oriented $(n-1)$-dimensional sphere $S^{n-1}$ into the $n$-dimensional sphere $ S^n$.
 "Tame" here means just "not wild", so it is sufficient to consider just piecewise linear embeddings.

  Each $(n-1)$-pseudosphere $E$ divides $ S^n$ into two parts, $E^+$ and $E^-$. We call them hemispheres related to $E$. Here "plus" and "minus" are  assigned consistent with the orientation of $E$.

  A pseudosphere arrangement on $S^n$ is a finite collection of $(n-1)$-pseudospheres that intersect along pseudospheres. That is,\begin{enumerate}
       \item Any number of pseudospheres from the arrangement intersect by some other pseudosphere.
       \item Any number of  (closed)
       hemispheres $E_i^+$ and $E_i^-$, where $E \in \mathcal{A}$, intersect by a topological ball.
     \end{enumerate}

\end{definition}

An oriented matroid is a concept which abstracts
 combinatorial properties of directed graphs,
point configurations,  vector configurations,   sphere arrangements, etc.
It is defined axiomatically, but we prefer not to present here the complete definition, referring the reader to \citep{AB61}.
The reason is that all what we need in the framework of the paper, is the following crucial
feature  of matroids,  the Folkman-Lawrence topological representation theorem:

\textit{Oriented matroids of rank $n$
are in a one-to-one correspondence with arrangements of $(n-1)$-pseudospheres.}

\bigskip

Here  are some further facts  about matroids:
\begin{enumerate}
    \item  Any configuration of spheres is automatically a pseudosphere arrangement, and therefore, represents
some oriented matroid.

    \item Some of pseudosphere arrangements can be \textit{straightened}, that is,
there exists a combinatorially equivalent arrangement of spheres.
Such arrangements represent the \textit{realizable matroids}.
    \item  However, there exist many non-realizable matroids. In other words, the class of pseudosphere arrangements is significantly
 wider than the class of sphere arrangements.

\end{enumerate}

\subsection*{ An oriented matroid with some extra properties generates
a collection of quasilinkages}
The below construction generalizes the walls-and-chambers stratification of
 the parameter space of polygonal linkages (see Section \ref{SectionIntroduction}).

For the classical setting, there exists just one parameter space $\mathbb{R}P^n_{>0}$
with a (unique) subdivision into chambers. However, for quasilinkages we have many different stratifications:
as explained below,  any matroid (with some extra properties) provides  an analogue of "parameter space + chambers".
The idea is  to replace the walls  $\sum_Ix_j=\sum_{\overline{I}}x_j$  by appropriate pseudospheres.
Besides, to single out the parameter space, we also need to replace coordinate
hyperplanes by some pseudospheres.

Assume we have an arrangement  $\mathcal{A}$ of
$(n+2^{n}-2)$  pseudospheres  on the sphere $S^{n-1}$.
Assume that $\mathcal{A}$ contains \begin{itemize}
                                     \item $n$ pseudospheres   $e_i$ labeled by the elements of $[n]$, and
                                     \item $(2^{n}-2)$ pseudospheres $E_I$ labeled by all proper non-empty subsets of the set $[n]$.
                                   \end{itemize}

Denote by  $\Delta$ the intersection of  the   hemispheres   associated to all the $e_i$
and to all  pseudospheres labeled by one-element sets : $$\Delta=\bigcap_{i=1}^n e_i^+ \cap\bigcap_{i=1}^n E_{\{i\}}^+.$$

\begin{definition} In this notation,
$\mathcal{A}$ is called a {\it Q-arrangement} if
the following holds:
\begin{enumerate}
\item
 Each subset $I$ and its complement
                     $\overline{I}=[n]\setminus I$  label one and the same pseudosphere, but with different orientations.
                     That is,
                     $$ E_{\overline{I}}^+=S^{n}\setminus E_I^+.$$

 \item  All the pseudospheres are different:
 for each $I\neq J\neq \overline{I}$, $E_I^{\pm}\neq E_J^{\pm}$
  \item For any  sets  $J \subseteq I \subset [n]$, we  always have

  $$E_I^+\cap \Delta \subseteq E_J^+.$$
\end{enumerate}
\end{definition}

Assume that a Q-arrangement $\mathcal{A}$ is fixed.
The pseudospheres from $\mathcal{A}$ tile the domain $\Delta$ into a number of (open) \textit{chambers}  separated
by
the intersections $E_I\cap \Delta$ that are called\textit{ walls}.
We say that two chambers $C,C'$ are \textit{adjacent} if there is exactly one wall separating them.

\begin{definition} Given a $Q$-arrangement, we associate with
each chamber $C$  a collection of short subsets   $L(C)$ of  the set $[n]$  by the following rule:

A subset $I\subset [n]$  is short whenever $C \subset E_I^+$.
\end{definition}

The following theorem follows straightforwadly from the above constructions.

\begin{theorem}Given a Q-arrangement,
\begin{enumerate}
                 \item By the above rule, each chamber $C$ yields a quasilinkage $L(C)$,
                 and, consequently, the PL manifold $M(L(C))$.
                \item The quasilinkages for two adjacent cameras differ by a flip.
                 \item The manifolds $M(L(C))$ and $M(L(C'))$  for two adjacent cameras
                   differ on a Morse surgery which is compatible to the
                 cell structure.
               \end{enumerate}
               \qed
\end{theorem}

\begin{figure}
\centering
\includegraphics[width=10 cm]{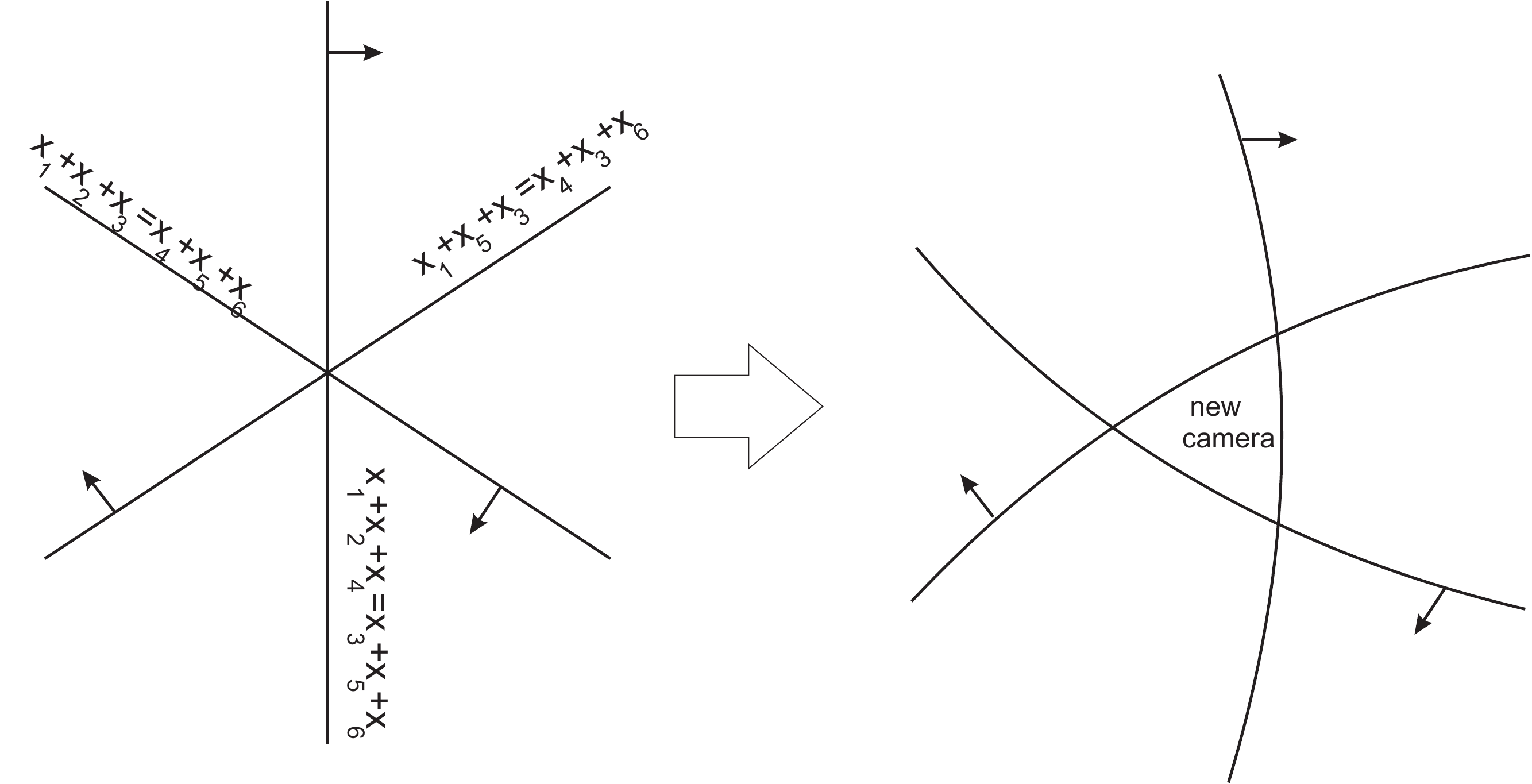}
\caption{The arrangement of pseudospheres generates the symmetric
quasilinkage from Example \ref{ex6}
\label{FigMatroidSmall}}
\end{figure}

\begin{example}
Consider the  collection walls and cameras  with $n=6$  for the
classical setting. Take the 10 walls  of type
$x_{i_1}+x_{i_2}+x_{i_3}=x_{i_4}+x_{i_5}+x_{i_6}$. They  intersect
at a single point $X=(1/2,1/2,...,1/2)$, and no other wall contains
the point $X$. We turn the walls to pseudospheres by a local
perturbation in a neighborhood of  $X$  in such a way that there
arises a new camera corresponding to the symmetric quasilinkage from
Example \ref{ex6}. Figure \ref{FigMatroidSmall}  gives an
illustration of the idea (however, in the figure we present a
smaller number of walls in the smaller dimension).
\end{example}

\section{Acknowledgements}
This research is supported  by JSC ``Gazprom Neft'' and by the Chebyshev Laboratory  (Department of Mathematics and Mechanics, St. Petersburg State University)  under RF Government grant 11.G34.31.0026



\end{document}